\documentclass[12pt,reqno,a4paper]{amsart}
\usepackage{graphicx,caption,amssymb,bbm}
\usepackage{hyperref}

\begin{document}

\let\kappa=\varkappa
\let\eps=\varepsilon
\let\phi=\varphi
\let\p\partial
\let\lle=\preccurlyeq
\let\ulle=\curlyeqprec

\def\Z{\mathbb Z}
\def\R{\mathbb R}
\def\C{\mathbb C}
\def\Q{\mathbb Q}
\def\P{\mathbb P}
\def\HH{\mathrm{H}}
\def\XX{\mathcal{X}}

%% small mathbb using bbm package
\def\bbk{\mathbbm{k}}

\def\conj{\overline}
\def\Beta{\mathrm{B}}
\def\const{\mathrm{const}}
\def\ov{\overline}
\def\wt{\widetilde}
\def\wh{\widehat}

\renewcommand{\Im}{\mathop{\mathrm{Im}}\nolimits}
\renewcommand{\Re}{\mathop{\mathrm{Re}}\nolimits}
\newcommand{\codim}{\mathop{\mathrm{codim}}\nolimits}
\newcommand{\Aut}{\mathop{\mathrm{Aut}}\nolimits}
\newcommand{\lk}{\mathop{\mathrm{lk}}\nolimits}
\newcommand{\sign}{\mathop{\mathrm{sign}}\nolimits}
\newcommand{\rk}{\mathop{\mathrm{rk}}\nolimits}

\def\id{\mathrm{id}}
\def\Leg{\mathrm{Leg}}
\def\Jet{{\mathcal J}}
\def\sS{{\mathcal S}}
\def\lcan{\lambda_{\mathrm{can}}}
\def\ocan{\omega_{\mathrm{can}}}
\def\bgamma{\boldsymbol{\gamma}}

\def\Sph{\mathrm{S}}
\def\SN{\Sigma_n}

\renewcommand{\mod}{\mathrel{\mathrm{mod}}}

\newtheorem{mainthm}{Theorem}
\renewcommand{\themainthm}{{\Alph{mainthm}}}
\newtheorem{thm}{Theorem}[section]
\newtheorem{lem}[thm]{Lemma}
\newtheorem{prop}[thm]{Proposition}
\newtheorem{cor}[thm]{Corollary}

\theoremstyle{definition}
\newtheorem{exm}[thm]{Example}
\newtheorem{rem}[thm]{Remark}
\newtheorem{df}[thm]{Definition}
\newtheorem{que}[thm]{Question}
\newtheorem{conje}[thm]{Conjecture}

\numberwithin{equation}{section}
\newcommand{\aftersubsec}{\hfill\nopagebreak\par\smallskip\noindent}

\title{Unknotting Lagrangian $\Sph^1\times\Sph^{n-1}$ in $\R^{2n}$}
\author[S.~Nemirovski]{Stefan Nemirovski}
\thanks{This work was partially supported by the SFB TRR 191 {\it Symplectic Structures in Geometry, Algebra and
Dynamics}, funded by the DFG (Projektnummer 281071066~--~TRR 191).}
\address{%
Steklov Mathematical Institute, Gubkina 8, 119991 Moscow, Russia;\hfill\break
\phantom{\& }Fakult\"at f\"ur Mathematik, Ruhr-Universit\"at Bochum, 44780 Bochum, \phantom{\& }Germany}
\email{stefan@mi-ras.ru}

\begin{abstract}
Lagrangian embeddings $\Sph^1\times\Sph^{n-1}\hookrightarrow\R^{2n}$ 
are classified up to smooth isotopy for all $n\ge 3$.
\end{abstract}

\maketitle

\section{Introduction and overview}
Two embeddings of a manifold $N$ in a manifold $M$ are isotopic
if they can be connected by a family of embeddings.
Diffeomorphisms of $N$ act on the set of isotopy classes
of embeddings $N\hookrightarrow M$ by precomposition.
The quotient by this action is the set of isotopy classes
of (unparametrised) submanifolds in $M$ diffeomorphic to~$N$.

The purpose of this paper is to describe the isotopy classes
represented by {\it Lagrangian\/} embeddings of $\Sph^1\times\Sph^{n-1}$
in~$(\R^{2n},\omega_{\mathrm{st}})$. 
We assume $n\ge 2$ to exclude the elementary special case $n=1$.
$\R^{2n}$ is identified with $\C^n$ so that the
symplectic form is the imaginary part of the
standard Hermitian metric. An embedding $j:N\to\C^n$ is Lagrangian
if multiplication by $i=\sqrt{-1}$ maps the tangent
space of $j(N)$ at each point onto its Euclidean orthogonal complement.
Lagrangian embeddings are obviously {\it totally real},
that is, have no tangent complex lines.

An explicit example of a Lagrangian embedding $\Sph^1\times\Sph^{n-1}\hookrightarrow\C^n$ 
for every $n$ is given by
\begin{equation}
\label{Chekanov}
\Sph^1\times\Sph^{n-1}\ni\; (\theta,x_1,\dots,x_n) \longmapsto \bigl(1+\frac{1}{2} e^{i\theta}\bigr) (x_1,\dots,x_n)\;\in\C^n,
\end{equation}
where we are viewing $\Sph^1$ as $\R/2\pi\Z$ and $\Sph^{n-1}$ as the unit sphere in~$\R^n$.

Our main Theorem~\ref{mainemb} shows that for $n\ne 2,4$ every Lagrangian embedding
$\Sph^1\times\Sph^{n-1}\hookrightarrow\C^n$ is isotopic
to~\eqref{Chekanov} if $n$ is even 
and either to~\eqref{Chekanov} or  to its composition with the reflection 
$(\theta,x)\mapsto (-\theta,x)$ if $n$ is odd. 
The isotopy is {\it not\/} through Lagrangian embeddings in general. 
In fact, there exist Lagrangian embeddings 
$\Sph^1\times\Sph^{n-1}\hookrightarrow\C^n$ for all $n>2$ 
that are not isotopic through Lagrangian embeddings 
to any reparametrisation of~\eqref{Chekanov}, 
see~\cite[\S 4$^\circ$]{Po}, \cite[Corollary 1.6]{EEMS}, and~\cite[\S 4]{PS}.

\begin{thm} 
\label{mainemb}
If $n\ge 6$ is even, then all Lagrangian embeddings
of\/ $\Sph^1\times\Sph^{n-1}$ in $\C^n$ are smoothly isotopic.
If $n\ge 3$ is odd, then Lagrangian embeddings of\/ $\Sph^1\times\Sph^{n-1}$ in $\C^n$
fall into two distinct smooth isotopy classes related by precomposition with 
a reflection of\/~$\Sph^1$.
\end{thm}

The result for $n=8$ was obtained by Borrelli~\cite{Bo};
our argument is a generalisation of his approach based on~\cite[\S 6]{Ne}.
Dimitroglou Rizell and Evans~\cite[Corollary E]{DE} proved the result 
for Lagrangian embeddings with minimal Maslov number~$n$ 
for even~$n\ge 6$ and for monotone Lagrangian embeddings for odd~$n\ge 5$.
However, embeddings~\eqref{Chekanov} have minimal Maslov number~$2$
for all~$n$ and non-monotone Lagrangian embeddings $\Sph^1\times\Sph^{2k}\to\C^{2k+1}$
exist for all $k\ge 1$ by~\cite[Corollary~1.6]{EEMS}.

\smallskip
The two cases $n=2$ and $n=4$ excluded from the theorem 
are of a rather different nature:

\smallskip
\noindent
$\bullet\quad$If $n=4$, then there are two smooth isotopy classes 
of embeddings (see below) and both contain Lagrangian
representatives~\cite{Bo}. This result falls within
the scope of the present paper and the exceptional
behaviour has a homotopy theoretic explanation,
see Remark~\ref{dim4}.

\smallskip
\noindent
$\bullet\quad$If $n=2$, then all Lagrangian $2$-tori in $\C^2$ are 
isotopic through Lagrangian tori by~\cite{DGI} and
a classification of Lagrangian embeddings up to smooth or Lagrangian
isotopy can be {\it a~posteriori\/} derived from~\cite{Ya}.
The argument in~\cite{DGI} does not rely on a classification
of smooth embeddings because no such classification is
known (or expected to become known) in this dimension.

\begin{cor}
\label{mainsubm}
If $n\ne 4$, then all Lagrangian submanifolds  in $\C^n$ diffeomorphic to\/ $\Sph^1\times\Sph^{n-1}$
are smoothly isotopic to the image of \eqref{Chekanov}.
\end{cor}

The extent to which Theorem~\ref{mainemb} and Corollary~\ref{mainsubm}
restrict the differential topology of Lagrangian embeddings and submanifolds
for $n\ge 3$ can be seen from the following brief summary of embedding theory:

\smallskip
\noindent
$\bullet\quad$For even $n\ge 4$, the isotopy classes 
of embeddings $\Sph^1\times\Sph^{n-1}\hookrightarrow\C^n$ are in
one-to-one correspondence with $\Z/2\Z$ by~\cite[Theorem~(2.4)]{HH}. 
The action of diffeomorphisms by precomposition is trivial~\cite[\S 5.5]{Bo}, 
so the isotopy classifications of embeddings and submanifolds coincide.
In~particular, Corollary~\ref{mainsubm} does indeed fail for~$n=4$.

\smallskip
\noindent
$\bullet\quad$For odd $n\ge 5$, the isotopy classes 
of embeddings $\Sph^1\times\Sph^{n-1}\hookrightarrow\nolinebreak\C^n$
are in one-to-one correspondence with $\Z$ by~\cite[Theorem~(2.4)]{HH}. 
The action of diffeomorphisms by precomposition 
corresponds to multiplication by $\pm 1$, 
which follows from the proof in~\cite[\S 5.5]{Bo} 
by using the first case of~\cite[Theorem~43]{Ba}, 
so there are infinitely many isotopy classes of submanifolds.

\smallskip
\noindent
$\bullet\quad$For $n=3$, the situation is richer~\cite{Sk}.
There is a surjective map $W$ from the set of isotopy classes
of embeddings onto $\Z$ such that each fibre $W^{-1}(d)$, $d\in\Z$, 
is in one-to-one correspondence with~$\Z/|d|\Z$. The action
of diffeomorphisms by precomposition does not seem to appear 
explicitly in the literature; anyhow, the diffeotopy group 
of $\Sph^1\times\Sph^2$ is finite (and
isomorphic to $(\Z/2\Z)^3$, see e.g.~\cite{DG}), so
there are infinitely many isotopy classes of 
submanifolds in this dimension as well.

\smallskip
Theorem~\ref{mainemb} is a typical `symplectic rigidity' result.
It becomes completely false if Lagrangian embeddings are
replaced by `soft' totally real embeddings in~$\C^n$.
Namely, {\it every embedding $\Sph^1\times\Sph^{n-1}\hookrightarrow\C^n$ 
is smoothly isotopic to a totally real embedding}. To see this, 
recall that any two embeddings of an orientable $n$-manifold
in $\C^n$ are regularly homotopic through immersions.
(This follows from the general theory in~\cite{Hi}.)
Therefore the tangent map of any embedding of $\Sph^1\times\Sph^{n-1}$ is homotopic
to the totally real tangent map of an (existing)
Lagrangian embedding. An isotopy to a totally real
embedding is then given by the $h$-principle~\cite[\S 27.4]{CEM}. 

\smallskip
The proof of Theorem~\ref{mainemb} follows the scheme from~\cite{Bo}
and consists of two parts discussed
in Section~\ref{luttlink} and Section~\ref{class}, respectively:

\smallskip
\noindent
(1) For a totally real embedding $j:\Sph^1\times\Sph^{n-1}\to\C^n$
and a nowhere vanishing vector field $v$ tangent to the $\Sph^1$-fibres,
we may consider the pushoff of $j$ in the direction of~$i\cdot dj(v)$.
If the embedding is Lagrangian, this pushoff is nullhomologous 
in $\C^n\setminus j(\Sph^1\times\Sph^{n-1})$ or, equivalently,
is unlinked with $(n-1)$-cycles in $\Sph^1\times\Sph^{n-1}$
up to replacing $v$ with $-v$ for odd~$n$. 
The proof of this assertion uses a generalisation of the Luttinger surgery~\cite{Lu} 
to arbitrary dimensions in the same way as the original surgery was used in~\cite{EP}.
Such a surgery produces (for $n\ge 3$) an exact symplectic manifold that is symplectomorphic 
to $(\R^{2n},\omega_{\mathrm{st}})$ outside of a compact subset
and that must therefore have the homology of $\R^{2n}$ by a basic application
of pseudoholomorphic curves~\cite{McD}. 
Comparing this with the homology calculations in the explicit model from~\cite[\S 6]{Ne}
proves the vanishing of the relevant linking numbers.

\smallskip
\noindent
(2) Embeddings of a connected orientable $n$-manifold $N$ in $\R^{2n}$ for $n\ge 4$
are governed by an $h$-principle due to Haefliger and Hirsch~\cite{HH}. 
Isotopy classes of such embeddings
are in one-to-one correspondence with homotopy classes
of equivariant maps from the frame bundle of~$N^\circ$ 
to the Stiefel manifold~$\mathcal{V}_{2n,n+1}$,
where $N^\circ$ is the complement to a small ball in~$N$.
The $(n+1)$-st component of the map to~$\mathcal{V}_{2n,n+1}$ is given by
a normal vector field on $N$ such that the pushoff 
of $N$ is {\it unlinked\/} with~$N^\circ$. 
Thus, for a Lagrangian embedding of $N=\Sph^1\times\Sph^{n-1}$,
this normal field can be taken to be $i\cdot dj(\pm v)$ by part~(1).
(The sign only makes a difference for odd $n$ 
and can be adjusted by a reflection of~$\Sph^1$.)
Recalling now that our manifold is parallelisable 
and our embedding is totally real, we can reduce the
frame map to a map from $N^\circ$ itself of the form
$x\mapsto \Phi(x)\beta(x)$, where $\Phi(x)\in\mathrm{GL}(n,\C)$ 
and $\beta$ is a fixed map to the Stiefel manifold.
Homotopy calculations (conveniently done in~\cite{DE})
show that for $n\ne 2, 4$ the map $\Phi$ has no effect on
the homotopy class, and the result for $n>4$ follows.
In the remaining case $n=3$, a version of the 
same argument in the setting of~\cite{Sk} 
shows that the Whitney invariant $W$ equals $\pm 1$
so that $\Z/|W|\Z=\{0\}$ and there is only one possible isotopy class
up to a reflection of~$\Sph^1$.

\section{Generalised Luttinger surgery, symplectic rigidity, and the linking class}
\label{luttlink}

Represent $\SN:=\Sph^1\times\Sph^{n-1}$, $n\ge 2$, as the quotient of $\R^n\setminus\{0\}$ by the $\Z$-action generated by the transformation
\begin{equation}
\label{act}
x\longmapsto 2x.
\end{equation}
The cotangent bundle $T^*\SN$ with its canonical symplectic structure 
is the quotient of $T^*(\R^n\setminus\{0\})\cong(\R^n_x\setminus\{0\})\times\R^n_y$
with the symplectic form $\sum dx_k\wedge dy_k$ by the symplectic $\Z$-action generated by
\begin{equation}
\label{act1}
(x,y)\longmapsto (2x,\frac{1}{2}y)
\end{equation}
in the `complex' coordinates $x=q$ and $y=-p$ on $T^*\R^n$.

The Riemannian metric 
$$
g=\frac{1}{\|x\|^2}{\sum dx_k^2}
$$ 
on $\R^n\setminus\{0\}$ is invariant with respect to~(\ref{act}) and therefore induces a metric on~$\SN$.
($\SN$ with this metric is the Riemannian product of the standard unit sphere in~$\R^n$ and the circle of length $\log 2$.)
Note further that the bundle $S_r^*(\R^n\setminus\{0\})\subset T^*(\R^n\setminus\{0\})$ of cotangent spheres
of radius $r>0$ with respect to $g$ is the hypersurface 
$$
\bigl\{\|x\|^2\|y\|^2=r^2 \bigr\} \subset (\R^n_x\setminus\{0\})\times (\R^n_y\setminus\{0\}).
$$

On $T^*(\R^n\setminus\{0\})$ with the zero section removed, consider the map
\begin{equation}
\label{map}
(x,y)\longmapsto (-y,x).
\end{equation}
Obviously, this map is a symplectomorphism preserving $S_r^*(\R^n\setminus\{0\})$ for every $r>0$. 
Furthermore, it maps the orbits of the action~(\ref{act1}) into orbits.
Hence, it defines a symplectomorphism $f$ of $T^*\SN$ with the zero section removed which maps $S_r^*\SN$ into itself
for every~$r>0$.

Similarly, the $g$-isometry 
\begin{equation}
\label{tau}
x\longmapsto\frac{x}{\|x\|^2}
\end{equation}
acts on $\SN=\Sph^1\times\Sph^{n-1}$ by a reflection of~$\Sph^1$
and its co-differential defines a symplectomorphism $\tau$ of the entire $T^*\SN$ preserving $S_r^*\SN$.

Let $j:\SN\to\C^n$ be a Lagrangian embedding and identify $\SN$ with the zero section of $T^*\SN$. 
By the Weinstein tubular neighbourhood theorem,
$j$ extends to a symplectomorphism from the disc bundle $D_r^*\SN$ with $r>0$ small enough 
onto a closed tubular neighbourhood $\conj{U}$ of $L=j(\SN)\subset\C^n$.
For any diffeomorphism $h$ of $S_r^*\SN\cong\p U$ from the subgroup generated by $f$ and $\tau$, 
consider the manifold
$$
X=X(j,h):=\conj{U} \cup_h \C^n\setminus U
$$
with the symplectic form $\omega_X$ equal to the standard symplectic form of $\C^n$
on $\conj{U}$ and $\C^n\setminus U$. This form is well-defined because $h$ extends 
to a symplectomorphism of~$\conj{U}\setminus L$.

\begin{lem}
\label{real}
The symplectic manifold $(X,\omega_X)$ is exact for $n\ge 3$. 
\end{lem}

\begin{proof}
Consider the following exact sequence in compactly supported cohomology 
$$
\dots\to \HH^{2n-2}_c(X\setminus L;\R)\to \HH^{2n-2}_c(X;\R)\to\HH^{2n-2}_c(L;\R)\to\dots
$$
The last term vanishes for $n\ge 3$ and hence the first map is surjective. 
$X\setminus L$ is symplectomorphic to $\C^n\setminus L$, which is exact,
so $[\omega_X]$ pairs trivially with $\HH^{2n-2}_c(X\setminus L;\R)$.
Thus, $[\omega_X]$ pairs trivially with $\HH^{2n-2}_c(X;\R)$ and therefore $\omega_X$ is exact
by Poincar\'e duality in cohomology.
\end{proof}

\begin{rem}
For $n\ge 4$, the sequence
$$
0\cong \HH^{2n-3}_c(L;\R) \to \HH^{2n-2}_c(\C^n\setminus L;\R)\to \HH^{2n-2}_c(\C^n;\R)\cong 0
$$
shows that $\HH^{2n-2}_c(X\setminus L;\R)=\HH^{2n-2}_c(\C^n\setminus L;\R)=0$ and
then the sequence in the proof of the lemma implies $\HH^{2n-2}_c(X;\R)=0$.
Hence, $\HH^2(X;\R)=0$ by duality, and this conclusion remains valid
for any orientable Dehn surgery on $L$ as in~\cite[Lemma~9]{Ne}.
Thus, the argument using the symplectic structure on $X$ is only needed for $n=3$.
\end{rem}

\begin{prop}
\label{rigid}
$X$ is diffeomorphic to $\C^n$.
\end{prop}

\begin{proof}
For $n\ge 3$, this is a consequence of Lemma~\ref{real} 
and the following classical rigidity result of Eliashberg, Floer, and McDuff~\cite{El,McD}:
A~symplectically aspherical (e.g.\ exact) symplectic manifold 
symplectomorphic to $(\C^n,\omega_{\mathrm{st}})$ outside
of a compact subset is diffeomorphic to~$\C^n$. 
For $n=2$, Luttinger~\cite[p.~222]{Lu} observed that $X$ is even symplectomorphic 
to $(\C^2,\omega_{\mathrm{st}})$.
\end{proof}

\begin{rem}
We shall actually use the weaker assertion that $X$ has the homology
of $\C^n$ for $n\ge 3$. This follows directly from \cite[\S 3.8]{McD}
and Lemma~\ref{real}.
\end{rem}

Take the orientation on $\R^n$ by the order of the coordinates 
and orient $T^*\R^n$ and $\C^n$ as $\R^n\oplus i\R^n$. 
(This orientation differs from the standard complex/symplectic
one by the sign $(-1)^{n(n-1)/2}$, which will make 
some of the following formulas nicer.)

In the $(n-1)$-dimensional homology group $\HH_{n-1}(S_r^*\SN;\Z)$,
consider the classes of the fibre
$$
\delta:=\bigl[\{\|y\|=r, x=(1,0,\dots,0)\}\bigr]
$$
and of the `horizontal' sphere
$$
\gamma:=\bigl[\{\|x\|=1, y=(r,0,\dots,0)\}\bigr]
$$
oriented as the boundaries of the corresponding balls in $\R^n$.

\begin{lem}
\label{homact}
Let $f$ and $\tau$ be the diffeomorphisms of $S_r^*\SN$ defined above.
Then
$$
\left\{
\begin{array}{rcl}
f_*\delta&=&(-1)^n\gamma\\
f_*\gamma&=&\delta
\end{array}
\right.
\quad
\text{ and }
\quad
\left\{
\begin{array}{rcl}
\tau_*\delta&=&-\delta\\
\tau_*\gamma&=&\gamma + (1+(-1)^n)\delta
\end{array}
\right.
$$
and the same identities hold in homology with coefficients in $\Z/p\Z$.
\end{lem}

\begin{proof}
The first three formulas follow immediately from the definitions
and the fact that the antipodal map reverses the orientation
on $\Sph^{n-1}$ if and only if $n$ is odd.
To get the formula for $\tau_*\gamma$, note that
$$
\begin{array}{rcl}
\tau_*(\gamma+\delta)&=& \bigl[ \tau\bigl(\{\|x\|=1, y=rx\}\bigr)\bigr]\\[2pt]
&=& \bigl[ \{\|x\|=1, y=-rx\}\bigr]\\[2pt]
&=& \gamma+(-1)^n\delta
\end{array}
$$
and combine this with $\tau_*\delta=-\delta$.
\end{proof}

Recall that if $N\subset\C^n$ is a connected oriented closed submanifold of dimension $k$,
then the linking number 
$$
\lk(N,\cdot):\HH_{2n-1-k}(\C^n\setminus N;\Z)\to \Z
$$
is an isomorphism. Furthermore, its reduction $\mod p$ defines an isomorphism
$\HH_{2n-1-k}(\C^n\setminus N;\Z/p\Z)\overset{\cong}{\longrightarrow} \Z/p\Z$ for every~$p$. 

\begin{prop}
\label{linkgamma}
Let $j:\SN\overset{\cong}{\longrightarrow} L\subset \C^n$ be a Lagrangian embedding
and let\/ $\wt\jmath:S_r^*\SN\overset{\cong}{\longrightarrow} \p U\subset \C^n$ 
be the induced embedding of the boundary of a Weinstein neighbourhood~$U$. 
Then $\lk(L,\wt\jmath_*\gamma)=\pm 1$ and moreover $\lk(L,\wt\jmath_*\gamma)=-1$ if $n$ is even
and the orientations are chosen as above.
\end{prop}

\begin{proof}
If $\lk(L,\wt\jmath_*\gamma)=0\mod p$, then $-(-1)^n \wt\jmath_*\gamma$ 
bounds an $n$-chain $\mod p$ in $\C^n\setminus U$.
This chain and the $n$-ball bounded by $\wt\jmath_*\delta$ in $\conj{U}$ 
form an $n$-cycle $\mod p$ in $X=X(j,f)=\conj{U}\cup_f (\C^n\setminus U)$
because $f_*\delta=(-1)^n\gamma$ by Lemma~\ref{homact}. 
The intersection index of $L$ with this cycle in $X$ is equal to $1$ 
and hence $L$ represents a non-trivial class in $\HH_n(X;\Z/p\Z)$, 
which contradicts Proposition~\ref{rigid}. So $\lk(L,\wt\jmath_*\gamma)$
is non-zero mod {\it every\/} $p$, which means that it equals~$\pm 1$.

If $n$ is even, we may also consider $X=X(j,\tau f)$. Then $\tau_* f_*\delta = \gamma+2\delta$
by Lemma~\ref{homact} and, taking $p=3$ in the preceding argument, we obtain 
$$
0\ne\lk(L,\wt\jmath_*(\gamma+2\delta))\mod 3 =\lk(L,\wt\jmath_*\gamma)+2 \mod 3,
$$
which excludes the value $1$.
\end{proof}

Let $j:N\to\C^n$ be a totally real (e.g.\ Lagrangian) embedding of a compact manifold
and $\zeta\in\Gamma(TN)$ a nowhere vanishing vector field on~$N$. 
The {\it $\C$-normal pushoff\/} of $j$ by $\zeta$ is the isotopy class in $\C^n\setminus j(N)$ 
of the embedding
\begin{equation}
\label{cpushof}
j^\sharp_\zeta(a) = j(a) + \eps i\cdot dj_a(\zeta(a)),\quad a\in N,
\end{equation}
for a sufficiently small $\eps>0$. If $N$ is closed and oriented, then the homology class 
\begin{equation}
\label{linkingclass}
\sigma(j,\zeta):=[j^\sharp_\zeta(N)]\in \HH_n(\C^n\setminus j(N);\Z)
\end{equation}
is called the {\it linking class\/}  of $j$ and~$\zeta$, cf.~\cite[\S 2]{Bo}.

If $\psi$ is a diffeomorphism of $N$, then
$$
(j\circ\psi)^\sharp_\zeta(a) = j(\psi(a)) + \eps i\cdot dj_{\psi(a)}(d\psi_a(\zeta(a))) = (j^\sharp_{\psi_*\zeta}\circ\psi)(a)
$$
and therefore
\begin{equation}
\label{linkclassreparam}
\sigma(j\circ\psi,\zeta) = \pm\sigma(j,\psi_*\zeta),
\end{equation}
where the sign depends on whether $\psi$ preserves or reverses the orientation on~$N$.

\begin{thm}
\label{linkv}
Let $j:\SN\to\C^n$ be a Lagrangian embedding and take the nowhere vanishing 
vector field $v=\sum x_k\frac{\p}{\p x_k}$ on $\SN=\left(\R_x^n\setminus\{0\}\right)/\Z$. 
\begin{itemize}
\item[1)] If $n$ is even, then $\sigma(j,v)=0$.
\item[2)] If $n$ is odd, then either $\sigma(j,v)=0$ or $\sigma(j,-v)=0$.
\end{itemize}
\end{thm}

\begin{rem}
This theorem extends to arbitrary $n\ge 2$ the results 
of Eliashberg--Polterovich~\cite{EP} for $n=2$ and
Borrelli~\cite{Bo} for $n=4, 8$.
\end{rem}

\begin{proof}
Let $m$ denote the meridional $(n-1)$-sphere $\{\|x\|=1\}\subset\SN$.
In our model for the Weinstein neighbourhood of $L=j(\SN)$, the $\C$-normal 
pushoffs $j^\sharp_{\pm v}(m)$ of $m$ by $\pm v$ are represented by $\{\|x\|=1,y=\pm r x\}$, 
so 
$$
\bigl[j^\sharp_{\pm v}(m)\bigr] = \wt\jmath_*(\gamma + (\pm 1)^n\delta) \in \HH_{n-1}(\C^n\setminus L;\Z).
$$ 
Now we can apply Proposition~\ref{linkgamma}:

\begin{itemize}
\item If $n$ is even, then 
\begin{equation}
\label{linkeven}
\lk\bigl(L,[j^\sharp_{-v}(m)]\bigr)=\lk\bigl(L,\wt\jmath_*(\gamma+\delta)\bigr)=-1+1=0.
\end{equation}

\item If $n$ is odd, there are two cases. 
If $\lk(L,\wt\jmath_*(\gamma))=1$, then 
\begin{equation}
\label{linkodd+}
\lk(L,[j^\sharp_{-v}(m)])= \lk\bigl(L,\wt\jmath_*(\gamma-\delta)\bigr) =1-1=0.
\end{equation}
If $\lk(L,\wt\jmath_*(\gamma))=-1$, then 
\begin{equation}
\label{linkodd-}
\lk(L,[j^\sharp_{v}(m)])= \lk\bigl(L,\wt\jmath_*(\gamma+\delta)\bigr)=-1+1=0.
\end{equation} 
\end{itemize}

Note also that by \eqref{linkingclass}, \eqref{cpushof}, 
and the (anti-)commutation rule for the linking number (see e.g.\ \cite[p.~234]{DTG}), 
we have
\begin{equation}
\label{linkswap}
\begin{array}{rcl}
\lk\bigl(j(m),\sigma(j,w)\bigr)&=&\lk\bigl(j(m),[j^\sharp_w(\SN)]\bigr)\\[2pt]
&=& \lk\bigl(j^\sharp_{-w}(m),[j(\SN)]\bigr)\\[2pt] 
&=& \lk\bigl(L,[j^\sharp_{-w}(m)]\bigr)
\end{array}
\end{equation}
for any nowhere vanishing vector field $w$ on $\SN$.

Hence, $\sigma(j,v)$ is unlinked with $j(m)$ by~\eqref{linkswap} and~\eqref{linkeven} for even~$n$.
For odd $n$, either $\sigma(j,v)$ is unlinked with $j(m)$ by~\eqref{linkswap} and~\eqref{linkodd+}
or $\sigma(j,-v)$ is unlinked with $j(m)$ by~\eqref{linkswap} and~\eqref{linkodd-}.

If $n\ge3$, the linking number with $j(m)$ defines an isomorphism
$\HH_n(\C^n\setminus j(\SN);\Z)\cong\Z$, which completes the proof.

If $n=2$, then $\HH_2(\C^2\setminus j(\Sph^1\times\Sph^1);\Z)\cong\Z^2$ 
and we also have to consider the linking number with the
longitudinal circle~$\ell=\Sph^1\times\{\mathrm{pt}\}$. 
Let $\psi$ be the obvious diffeomorphism of $\Sph^1\times\Sph^1$
swapping $m$ and $\ell$. Applying the above argument to $j\circ\psi$,
we obtain
$$
0\overset{\eqref{linkswap},\eqref{linkeven}}{=}\lk\bigl(j\circ\psi(m),\sigma(j\circ\psi,v)\bigr)
\overset{\eqref{linkclassreparam}}{=}\lk\bigl(j(\ell),-\sigma(j,\psi_*v)\bigr).
$$
However, the vector field $\psi_*v\perp v$ is homotopic to $v$ through nowhere vanishing 
vector fields on $\Sph^1\times\Sph^1$ and therefore $\sigma(j,\psi_*v)=\sigma(j,v)$ by \eqref{linkingclass} and~\eqref{cpushof}.
Hence, $\sigma(j,v)$ is unlinked with $j(\ell)$ too.
\end{proof}

\begin{cor}
\label{linkreparam}
If $\sigma(j,v)\ne 0$ for a Lagrangian embedding $j:\SN\to\C^n$,
then $n$ is odd and $\sigma(j\circ \rho,v)=0$ for the 
diffeomorphism $\rho$ given by~\eqref{tau}.
\end{cor}

\begin{proof}
$\rho_*v=-v$ and the claim follows 
from~\eqref{linkclassreparam} and Theorem~\ref{linkv}.
\end{proof}

\section{Classification of embeddings}
\label{class}

Let us start with the special case $n=3$.
The isotopy classes of smooth embeddings $j:\Sigma_3\to\R^6\cong\C^3$ 
are indexed by the following two invariants~\cite[p.~648]{Sk}. 
First, the {\it Whitney invariant\/}
$$
W(j)\in \HH_1(\Sigma_3;\Z)\cong\Z,
$$
which is essentially the difference class between the
Haefliger--Hirsch maps~\eqref{HHmap} of $j$ and of the standard 
embedding in $\R^2\times\R^3\cong\R^5\subset\R^6$.
Second, the {\it Kreck invariant\/}
$$
\eta(j) \in \Z/|W(j)|\Z,
$$
where we have used the above identification $\HH_1(\Sigma_3;\Z)\cong\Z$.
The Kreck invariant generalises the Haefliger invariant~\cite{Ha}
classifying smoothly knotted $3$-spheres in~$\R^6$.

In particular and most importantly for us, the Kreck invariant 
of an embedding with $W(j)=\pm 1\in\Z$ takes values in $\{0\}$
and therefore all such embeddings with the same Whitney invariant are isotopic. 
(This case is singled out in Corollary~(a) on p.~649 of~\cite{Sk}.)  

The Bo\'{e}chat--Haefliger Invariant Lemma on p.~661 of~\cite{Sk}
asserts that if $\xi$ is an {\it unlinked section\/} of the normal bundle of~$j$, then
\begin{equation}
\label{Whitney}
2W(j)=\pm\mathop{\mathrm{PD}}\,e(\xi^\perp)\in\HH_1(\Sigma_3;\Z),
\end{equation}
that is, $2W(j)$ is Poincar\'e dual to the Euler class
of the orthogonal complement of $\xi$ in the normal bundle of~$j(\Sigma_3)$.
The sign $\pm$ in~\eqref{Whitney} is {\it fixed\/} by orientation
conventions as stated right before the Difference Lemma on p.~658 of~\cite{Sk}.

For a Lagrangian embedding $j$, an unlinked normal bundle section
is given by $\xi:=i\cdot dj(v)$ after possibly precomposing $j$
with the diffeomorphism~\eqref{tau} by Corollary~\ref{linkreparam}. 
Multiplication by $i$ defines an isomorphism $\xi^\perp\cong v^\perp$,
where the latter  complement is taken in~$T\Sigma_3$.
The orthogonal complement to $v$ is just the subbundle
tangent to the $\Sph^2$-fibres in $\Sigma_3$,
so its Euler class equals twice
a generator of $\HH^2(\Sigma_3;\Z)$. Hence,
$$
W(j)=\pm 1\in\Z\cong \HH_1(\Sigma_3;\Z)
$$
by formula~\eqref{Whitney}, where the sign is again fixed
by the orientation conventions and the identification $\HH_1(\Sigma_3;\Z)\cong\Z$. Now the above remark about 
the Kreck invariant implies the following result for $n=3$.

\begin{prop}
\label{dim3}
All Lagrangian embeddings $\Sph^1\times\Sph^2\hookrightarrow\C^3$ 
are smoothly isotopic up to precomposition with~\eqref{tau}.
\end{prop}

Let now $n\ge 2$ be arbitrary. Trivialise the tangent bundle of $\SN$ 
using the dilation invariant vector fields
$$
\zeta_k = \|x\|\frac{\p}{\p x_k}, \quad k=1,\ldots,n,
$$
on $\R^n\setminus\{0\}$ and note that 
$$
v=\sum\frac{x_k}{\|x\|}\zeta_k.
$$
Fix also a complex linear trivialisation of $T\C^n$ so that $T_a\C^n\cong\C^n$ 
for all $a\in\C^n$.

Let $j:\SN\to\C^n$ be an embedding and let $\SN^\circ$ denote the complement
of a small $n$-ball in~$\SN$. There exists a unique up to homotopy normal vector field
$\nu_j$ on $j(\SN)$ such that it does not vanish on $\SN^\circ$  
and the pushoff of $j(\Sigma_n)$ by $\nu_j$ is nullhomologous in $\C^n\setminus j(\SN^\circ)$,
see~\cite[p.~134]{HH}. (This vector field is called the associated vector field
of $j$ in~\cite[\S 5.1]{Bo} and the Haefliger--Hirsch vector field in~\cite[\S 6.2]{DE}.
It is also the unlinked section of the normal bundle from~\cite{Sk} appearing in~\eqref{Whitney}.)

Specialising \cite[Theorem (2.2)]{HH} and the remarks right after it to our situation,
we see that for $n\ge 4$ the isotopy classes of embeddings $j$ are in one-to-one
correspondence with the homotopy classes of their {\it Haefliger--Hirsch maps}
\begin{equation}
\label{HHmap}
\SN^\circ \ni x \longmapsto \bigl(dj_x(\zeta_1(x)),\dots,dj_x(\zeta_n(x)),\nu_j(j(x)\bigr) \in \mathcal{V}_{2n,n+1}
\end{equation}
to the Stiefel manifold of real $(n+1)$-frames in $\R^{2n}\cong\C^n$, cf.~\cite[\S 5.2]{Bo}. 
(It is convenient to use arbitrary frames here but it is well-known that $\mathcal{V}_{*,*}$
retract onto Stiefel manifolds of orthonormal frames~\cite[\S 2]{St}.)

Since $ \mathcal{V}_{2n,n+1}$ is $(n-2)$-connected~\cite[Satz 8]{St} 
and $\SN^\circ$ retracts onto the wedge sum $\Sph^1\vee\Sph^{n-1}$,
it follows that for $n\ge 3$ the homotopy class of~\eqref{HHmap} is determined 
by the image of the meridional $(n-1)$-sphere $m\subset\SN^\circ$
in~$\pi_{n-1}(\mathcal{V}_{2n,n+1})$. 
The latter group is isomorphic to $\Z$ if $n$ is odd 
and to $\Z/2\Z$ if $n$ is even~\cite[Satz 9]{St},
which completes the classification of smooth embeddings
of $\SN$ in $\C^n$ for $n\ge 4$.

Assume now that $j$ is a Lagrangian embedding. By Corollary~\ref{linkreparam},
up to precomposing $j$ with the reflection~\eqref{tau} for odd~$n$,
we may choose $\nu_j = i\cdot dj(v)$ so that the map~\eqref{HHmap}
takes the form
$$
x \longmapsto \left(dj_x(\zeta_1(x)),\dots,dj_x(\zeta_n(x)), i\sum\limits_{k=1}^{n} \frac{x_k}{\|x\|} dj_x(\zeta_k(x))\right).
$$
Furthermore, since a Lagrangian embedding is totally real, the vectors $dj_x(\zeta_k(x))$
form a complex frame of $\C^n$ for every $x\in\SN^\circ$. Therefore there is a smooth
map
$$
\Phi: \SN^\circ \longrightarrow \mathrm{GL}(n,\C)
$$
such that 
$$
dj_x(\zeta_k(x)) = \Phi(x)e_k,
$$
where $e_k$, $k=1,\ldots,n$, is our fixed complex frame of~$\C^n$. 
Hence, the map~\eqref{HHmap} can be written as
\begin{equation}
\label{HHmapLagr}
x \longmapsto \left(\Phi(x)e_1,\dots,\Phi(x)e_n, \Phi(x)\sum\limits_{k=1}^{n} \frac{x_k}{\|x\|} ie_k\right).
\end{equation}

The action of $\mathrm{GL}(n,\C)\subset \mathrm{GL}_+(2n,\R)$ on $\mathcal{V}_{2n,n+1}$ 
induces a trivial map on the relevant homotopy group $\pi_{n-1}$ for all $n\ne 2, 4$
by the proof of \cite[Lemma 6.14]{DE}. Specifically,
\begin{itemize}
\item[(1)] if $n$ is odd, then $\pi_{n-1}(\mathrm{GL}(n,\C))=\pi_{n-1}(\mathrm{U}(n))=0$;
\item[(2)] if $n\ne 2, 4, 8$ is even, then the map 
$$\pi_{n-1}(\mathrm{GL}_+(2n,\R))\longrightarrow \pi_{n-1}(\mathcal{V}_{2n,n+1})=\Z/2\Z
$$
induced by the action is trivial;
\item[(3)] if $n=8$, then the map in (2) is surjective but the map 
$$
\Z=\pi_{7}(\mathrm{GL}(8,\C))\longrightarrow \pi_{7}(\mathrm{GL}_+(16,\R))=\Z
$$ 
is the multiplication by~$2$, as observed already in~\cite[\S 5.3]{Bo}.
\end{itemize}
Thus, for $n\ne 2, 4$, the Haefliger--Hirsch map~\eqref{HHmapLagr} of any Lagrangian embedding 
(precomposed with~\eqref{tau} if $n$ is odd and $\sigma(j,v)\ne 0$)
is homotopic to the standard map
\begin{equation}
\label{HHstandard}
x \longmapsto \left(e_1,\dots, e_n, \sum\limits_{k=1}^{n} \frac{x_k}{\|x\|} ie_k\right),
\end{equation}
which proves the following result for all $n\ge 5$. 

\begin{prop}
\label{dimn}
All Lagrangian embeddings $\Sph^1\times\Sph^{n-1}\hookrightarrow\C^n$, $n\ge 5$,
are smoothly isotopic up to precomposition with~\eqref{tau} if $n$ is odd.
\end{prop}

\begin{rem}
\label{dim4}
If $n=4$, the following argument from~\cite[p.~232]{Bo} shows that 
both isotopy classes of embeddings $j:\Sigma_4\to\C^4$ have Lagrangian representatives.
Note first that (for any $n\ge 3$) the action of the map~$\Phi$ in~\eqref{HHmapLagr} on $\pi_{n-1}(\Sigma_n^\circ)$ 
can be chosen at will by~\cite[\S 6.2]{Au} and the $h$-principle for 
Lagrangian immersions of~$\Sph^{n-1}$~\cite[\S 24.3]{CEM}. However, the map 
$$
\Z=\pi_3(\mathrm{GL}(4,\C))\longrightarrow \pi_3(\mathcal{V}_{8,5})=\Z/2\Z
$$
is surjective and so each homotopy class of maps $\Sigma_4^\circ\to \mathcal{V}_{8,5}$
can be represented by the Haefliger--Hirsch map of a Lagrangian embedding. 
\end{rem}

\begin{rem}
\label{dim3alt}
For $n=3$, the proof of Proposition~\ref{dimn} amounts to 
an alternative computation of the Whitney invariant. 
Let $\mathcal{N}\to\mathcal{V}_{2n,n+1}$ be the
tautological normal bundle. (The fibre of $\mathcal{N}$
at a frame is the orthogonal complement to the span
of the frame.) An orientation on $\R^{2n}$ defines
an orientation on $\mathcal{N}$. Two observations are
immediate:
\begin{enumerate}
\item The pull-back of $\mathcal{N}$ by the
Haefliger--Hirsch map~\eqref{HHmap} is the orthogonal
complement $\nu_j^\perp(=\xi^\perp)$ in the normal bundle of~$j$.
\item The pull-back of $\mathcal{N}$ to the 
meridional $2$-sphere $m=\{\|x\|=1\}$
by the standard map~\eqref{HHstandard} is the tangent
bundle of~$m$.  
\end{enumerate}
For a Lagrangian embedding $j$ (precomposed with~\eqref{tau} if needed), 
the maps~\eqref{HHmap} and~\eqref{HHstandard} are homotopic. 
Hence, $\langle e(\xi^\perp),[m]\rangle = 2$
and therefore $W(j)=\pm 1\in \Z\cong\HH_1(\Sigma_3;\Z)$ by~\eqref{Whitney}.
\end{rem}

It remains to show that precomposition with the involution $\rho$ of $\SN$ 
defined by~\eqref{tau} does indeed change the smooth isotopy class of a Lagrangian 
embedding for odd~$n$. Swapping $j$ and $j\circ\rho$ if necessary, we may
assume that $\sigma(j,v)=0$ by Corollary~\ref{linkreparam} so that
$\sigma(j\circ\rho,-v)=0$ by~\eqref{linkclassreparam}. 
Hence, the argument leading to Proposition~\ref{dimn}
shows that the Haefliger--Hirsch maps of $j$ and $j\circ\rho$ are 
homotopic to
\begin{equation}
\label{HHstandardpm}
x \longmapsto \left(e_1,\dots, e_n, \pm \sum\limits_{k=1}^{n} \frac{x_k}{\|x\|} ie_k\right)
\end{equation}
with the $+$ sign for $j$ and the $-$ sign for $j\circ\rho$.
The restrictions of the maps~\eqref{HHstandardpm} to the meridional sphere $m=\{\|x\|=1\}$ 
parametrise the (homotopy) fibre $\Sph^{n-1}\sim \mathcal{V}_{n,1}\times\R^n$ of 
the `forgetful' fibre bundle 
$$
\mathcal{V}_{2n,n+1} \longrightarrow \mathcal{V}_{2n,n}
$$
and differ by the antipodal map of $\Sph^{n-1}$. For odd $n$, 
the antipodal map is orientation reversing and hence acts non-trivially 
on $\pi_{n-1}(\Sph^{n-1})=\Z$.
Furthermore, $\pi_{n-1}(\mathcal{V}_{2n,n})=0$ by~\cite[Satz 8]{St}
and $\pi_n(\mathcal{V}_{2n,n})=\Z/2\Z$ for odd $n>1$ by~\cite[Satz 9]{St},
so the homotopy exact sequence 
$$
\pi_n(\mathcal{V}_{2n,n}) \to \pi_{n-1}(\Sph^{n-1}) \to \pi_{n-1}(\mathcal{V}_{2n,n+1}) \to \pi_{n-1}(\mathcal{V}_{2n,n})
$$
shows that the embedding of the fibre is an isomorphism on $\pi_{n-1}$.
Hence, the two maps in~\eqref{HHstandardpm} are not homotopic, which proves the next statement.

\begin{prop}
\label{reflectodd}
A Lagrangian embedding $\Sph^1\times\Sph^{n-1}\hookrightarrow\C^n$ with odd~$n\ge 3$
is not smoothly isotopic to its composition with~\eqref{tau}.
\end{prop}

%% \begin{rem}
%% For $n=3$ one can alternatively (and in fact equivalently) use 
%% formula~\eqref{Whitney} and the specific choice of the unlinked section~$\xi$ 
%% in the proof of Proposition~\ref{dim3} to show that $W(j\circ\rho)=-W(j)$
%% for a Lagrangian embedding~$j$.
%% \end{rem}

Theorem~\ref{mainemb} follows by combining Propositions~\ref{dim3}, \ref{dimn}, and~\ref{reflectodd}.

\subsection*{Acknowledgements}
The author is grateful to Leonid Polterovich, Felix Schlenk, and Arkadiy Skopenkov
for their comments on the first draft of this paper.

\end{document}